\documentclass[11pt,oneside]{amsart}

\usepackage{amsmath,ifthen, amsfonts, amssymb,
srcltx, amsopn, color, enumerate, hyperref}
\usepackage[cmtip,arrow]{xy}
\usepackage{pb-diagram, pb-xy}
\usepackage{overpic}

\usepackage[T1]{fontenc}

\newcommand{\showcomments}{yes}

\newsavebox{\commentbox}
{\ifthenelse{\equal{\showcomments}{yes}}%
{\footnotemark
        \begin{lrbox}{\commentbox}
        \begin{minipage}[t]{1.25in}\raggedright\sffamily\tiny
        \footnotemark[\arabic{footnote}]}
{\begin{lrbox}{\commentbox}}}%
{\ifthenelse{\equal{\showcomments}{yes}}%
{\end{minipage}\end{lrbox}\marginpar{\usebox{\commentbox}}}
{\end{lrbox}}}

\newcounter{ax}
\newtheorem{thm}{Theorem}[section]

\newtheorem{lem}[thm]{Lemma}

\newtheorem{cor}[thm]{Corollary}

\theoremstyle{definition}

\newtheorem{rem}[thm]{Remark}

\newtheorem{claim*}{Claim}

\DeclareMathOperator{\rank}{rk}

\DeclareMathOperator{\Aut}{Aut}
\DeclareMathOperator{\Out}{Out}

\newcommand{\neb}{\mathcal N}

\newcommand{\field}[1]{\mathbb{#1}}
\newcommand{\integers}{\ensuremath{\field{Z}}}

\newcommand{\naturals}{\ensuremath{\field{N}}}
\newcommand{\reals}{\ensuremath{\field{R}}}

\makeatletter

\newcommand{\Rmnum}[1]{\mathbf{{\expandafter\@slowromancap\romannumeral #1@}}}

\makeatother

\let\oldmarginpar\marginpar
\renewcommand\marginpar[1]{\-\oldmarginpar[\raggedleft\footnotesize #1]%
{\raggedright\footnotesize #1}}

\newcounter{enumitemp}

\newcommand{\dist}{\textup{\textsf{d}}}

\newcommand{\growth}{\mathrm{GR}}

\begin{document}
\title{A remark on thickness of free-by-cyclic groups}
\date{\today}
\author[M. Hagen]{Mark Hagen}
\address{School of Mathematics, University of Bristol, Bristol, United Kingdom}
\email{markfhagen@posteo.net}

\keywords{thick space, wide space, relatively hyperbolic group, free group automorphism}
\maketitle

\begin{abstract}
Let $F$ be a free group of positive, finite rank and let $\Phi\in\Aut(F)$ be a polynomial-growth automorphism.  Then 
$F\rtimes_\Phi\integers$ is \emph{strongly thick} of order $\eta$, where $\eta$ is the rate of polynomial growth of $\phi$. 
 This fact is implicit in work of Macura~\cite{Macura}, but~\cite{Macura} predates the notion of thickness.  Therefore, in 
this note, we make the relationship between polynomial growth of and thickness explicit.  Our result 
combines with a result independently due to Dahmani-Li, Gautero-Lustig, and Ghosh to show that free-by-cyclic groups 
admit relatively hyperbolic structures with thick peripheral subgroups.
\end{abstract}

\section{Definitions, statement, discussion}
There has been significant interest in the geometry of mapping tori of polynomial-growth automorphisms of finite-rank free 
groups (see e.g.~\cite{Button,ButtonKropholler,Macura,BFH:Kolchin}).  There is also a considerable literature 
on hyperbolicity, relative hyperbolicity, and acylindrical hyperbolicity of mapping tori of general automorphisms of free 
groups.  For 
example, 
$F\rtimes_\Phi\integers$ is word-hyperbolic exactly when $\Phi\in\Aut(F)$ 
is atoroidal~\cite{Brinkmann,BestvinaFeighn:comb}, and recent work of Dahmani-Li~\cite{DahmaniLi} and 
Ghosh~\cite{Ghosh} characterises nontrivial relative hyperbolicity of $F\rtimes_\Phi\integers$: it is equivalent 
to exponential growth of $\Phi$.  Even in the polynomial-growth 
case, where nontrivial relative hyperbolicity is impossible (by combining~\cite[Theorem 7.2]{Macura} and~\cite[Theorem 
1.3]{Sisto}), recent results show that virtual acylindrical hyperbolicity holds provided $\Phi$ has infinite 
order~\cite{Ghosh,ButtonKropholler}.  In this note, we show that when $\Phi$ has polynomial growth, $F\rtimes_\Phi\integers$ 
is non-relatively hyperbolic in a strong way: 
$F\rtimes_\Phi\integers$ is \emph{thick} in the sense of~\cite{BDM:thick}.

There is a general question of which classes $\mathcal C$ of groups have the property that each 
$G\in\mathcal C$ is either relatively hyperbolic or thick, and, more strongly, which $\mathcal C$ have the property that 
each $G\in\mathcal C$ exhibits a (possibly trivial) relatively hyperbolic structure in which the peripheral subgroups are 
thick.  This property is interesting because such a relatively hyperbolic structure is quasi-isometry invariant and 
``minimal'': each peripheral subgroup is peripheral in any 
relatively hyperbolic structure on $G$, by~\cite[Corollary 4.7]{BDM:thick} or~\cite[Theorem 1.7]{DrutuSapir}.  

Classes of groups that 
have (possibly trivial) relatively hyperbolic structures with thick peripherals 
include Coxeter groups~\cite{BHS:cox}, fundamental groups of ``mixed'' 
$3$--manifolds (consider the geometric decomposition and 
apply~\cite[Theorem 1.2]{BDM:thick} to the graph manifold pieces), and Artin 
groups (combine~\cite[Lemma 10.3]{BDM:thick} with~\cite[Theorem 1.2]{CharneyParis}).

Our main result combines with a theorem established independently by Dahmani-Li, 
Gautero-Lustig, and Ghosh to yield:

\begin{cor}[Relatively hyperbolicity with thick 
peripherals]\label{cor:minimal_relhyp}
Let $F$ be a free group of finite positive rank, let $\Phi\in\Aut(F)$, and let $G=F\rtimes_\Phi\integers$.  Then either $G$ 
is thick, or $G$ is hyperbolic relative to a finite collection of proper subgroups, each of which is thick.
\end{cor}

\begin{proof}
If $\Phi$ is polynomially growing or of finite order, then $G$ is thick by Theorem~\ref{thm:main}.  
Otherwise, $\Phi$ is exponentially growing.  Theorem~3.9 of~\cite{DahmaniLi} implies that $G$ is hyperbolic relative to a 
finite collection $\mathcal P'$ of peripheral subgroups, each of which is the mapping torus of a polynomial-growth free 
group automorphism and therefore thick by Theorem~\ref{thm:main}.  (One can also use Corollary~3.12 of~\cite{Ghosh} in 
conjunction with~\cite[Theorem 4.8, Remark 7.2]{BDM:thick} in place of~\cite[Theorem 3.9]{DahmaniLi}.) 
\end{proof}

Combining Corollary~\ref{cor:minimal_relhyp} with~\cite[Theorem 4.8]{BDM:thick} shows that if $G'$ is a group quasi-isometric 
to a free-by-$\integers$ group, then $G'$ is hyperbolic relative to thick subgroups.

We now turn to our main theorem.  Fix a finite-rank free group $F$.  Given $\phi\in\Out(F)$, by a 
\emph{lift} $\Phi$ of $\phi$ we mean an automorphism $\Phi:F\to F$ whose outer class is $\phi$.  Fix a free basis $\mathcal 
S$ for $F$.  Recall that the \emph{growth function} $\growth_{\Phi,\mathcal S}:\naturals\to\naturals$ is defined by 
$\growth_{\Phi,\mathcal S}(n)=\max_{s\in\mathcal S}\|\Phi^n(s)\|$, where $\|g\|$ denotes word length.  Recall also that the 
asymptotic behaviour of 
$\growth_{\Phi,\mathcal S}$ is independent of the choice of generating set, and that the growth function is either 
exponential or polynomial of degree $\eta\leq |\mathcal S|$.  In the latter case, we say $\Phi$ (and its outer class $\phi$) 
are \emph{polynomially growing} and refer to $\eta$ as the \emph{polynomial growth rate}.

We now recall the notion of a \emph{thick} group, which was introduced in~\cite{BDM:thick} as both an obstruction to 
the existence of a nontrivial relatively hyperbolic structure and a ``structural'' version of the property of having a 
polynomial divergence function.  The definition of thickness is inductive, and, if $G$ is a thick group, there is an 
associated invariant $n\geq0$, the \emph{order of thickness}.  The reader is referred to~\cite{BDM:thick} 
and~\cite{BD:thick} for a more detailed discussion of the several closely-related notions of thickness.  Here, we just 
restate the facts about thickness needed for most of our discussion; see~\cite[Section 4]{BD:thick}.

\begin{itemize}
 \item A finitely generated group $G$ is \emph{strongly thick of order $0$} if no asymptotic cone of $G$ has a cut-point.  
For example, if $G\cong A\times B$, where $A,B$ are infinite groups, then $G$ is strongly thick of order $0$.
 \item Let $G$ split as a finite graph of groups where the edge groups are infinite and the vertex groups are thick 
of order $n$.  Suppose, moreover, that the vertex groups are \emph{quasi-convex}, in the sense that there exist constants 
$C,L$ so that for each vertex group $A$, any two points in $A$ can be connected by an $(L,L)$--quasigeodesic in 
$\neb_C(A)$.  Then $G$ is strongly thick of order $\leq n+1$.  (This is Proposition~4.4 in~\cite{BD:thick}.)
\end{itemize}

We will need the full definition of strong thickness in the case $n=1$, in the proof of Lemma~\ref{lem:linear_case}, 
so we give the definition in that proof.  Our main theorem is:

\begin{thm}\label{thm:main}
Let $F$ be a free group of finite rank at least $1$.  Let $\phi\in\Out(F)$ be polynomially-growing, with polynomial 
growth rate $\eta\ge0$, and let $\Phi\in\Aut(F)$ be a lift of $\phi$.  Then $F\rtimes_\Phi\integers$ is strongly 
thick of order $\eta$.
\end{thm}

If a group $G$ is thick of order $n$, then the divergence function of $G$ (see~\cite{DMS,BD:thick,Gersten1,Gersten2}) 
is bounded above by a polynomial of degree $n+1$, although lower bounds are more difficult to 
establish in general (see e.g.~\cite{DaniThomas,BHS:cox, Levcovitz1,Macura}).  In~\cite{Macura}, Macura gave upper and 
lower bounds on the divergence function of $F\rtimes_\Phi\integers$, both polynomial of degree 
$\eta+1$.  Macura's result uses the decomposition of $F\rtimes_\Phi\integers$ as graph of groups with $\integers$ edge 
groups coming from a relative train track representative for $\Phi$, and implies that $\eta$ distinguishes quasi-isometry 
types of mapping tori of polynomial-growth automorphisms.

To an extent, thickness is implicit in Macura's argument, but her work predates the formal definition by several years.  Our 
proof of Theorem~\ref{thm:main} follows Macura's strategy, relying on the same splitting coming from a relative train track 
representative.

\subsection*{Acknowledgments}
I am very grateful to Noel Brady for detailed discussions, and for explaining Macura's work on this subject, during a 
2014 
visit to the University of Oklahoma.  I am also grateful to Jason Behrstock, Pritam Ghosh, Ivan Levcovitz, Daniel 
Woodhouse, and the referee for several helpful comments and answers to questions; I thank the first three for encouraging me 
to write 
Theorem~\ref{thm:main} down.  This work was partly supported by EPSRC grant EP/R042187/1.

\section{Proof of Theorem~\ref{thm:main}}\label{sec:proof}
Throughout, we adopt the notation from Theorem~\ref{thm:main}.

\begin{proof}
Adopt the notation of the statement; in particular, $\Phi$ has polynomial 
growth rate of order $\eta$.  Since the order of strong thickness is a 
quasi-isometry invariant (see~\cite[Remark 7.2]{BDM:thick} and~\cite[Definition 
4.13]{BD:thick}), and the polynomial growth rate of any positive power of $\Phi$ coincides with that of $\Phi$, it suffices 
to prove the theorem for $G=F\rtimes_{\Phi^k}\integers$, for any $k>0$.  By~\cite[Theorem 5.1.5]{BFH:dynamics}, we can 
choose $k>0$ with the property that $\Phi^k$ admits an \emph{improved relative train track representative}.  In particular, 
there exists a finite connected graph $\Gamma$, with $\pi_1\Gamma$ identified with $F$, 
and a cellular map $f:\Gamma\to\Gamma$, inducing the map $\Phi^k$ on 
$\pi_1\Gamma$, so that the following hold:
\begin{enumerate}[(A)]
 \item \label{item:filtration}There is a filtration $\emptyset=\Gamma_0\subset\Gamma_1\subset\cdots\subset\Gamma_n=\Gamma$, 
with each 
$\Gamma_i$ an $f$--invariant subgraph.  Each vertex is fixed by $f$.
 \item \label{item:1_edge}For $1\le i\le n$, the graph $\Gamma_i$ is obtained from $\Gamma_{i-1}$ by adding an (oriented) 
edge $e_i$.
 \item \label{item:path}For each $i\ge1$, we have $f(e_i)=e_ip_i$, where $p_i$ is a closed 
edge-path whose edges belong to $\Gamma_{i-1}$.  
\end{enumerate}
We can take each $p_i$ to be immersed.  The latter two properties on the above 
list rely on the fact that our automorphism has polynomial growth rate.

An edge $e_i$ is \emph{invariant} if $f(e_i)=e_i$, i.e. if  $p_{i}$ is 
trivial.  

The graphs $\Gamma_i$ need not be connected when $i<n$.  More precisely, if 
$e_i$ is non-invariant, then $e_i$ necessarily shares a 
vertex with $\Gamma_{i-1}$.  However, if $e_i$ is invariant, then $e_i$ can be disjoint from 
$\Gamma_{i-1}$.  Hence, let $\Lambda_i$ index the set of components $\Gamma_i^\alpha,\alpha\in\Lambda_i$, of $\Gamma_i$. 

Note that since $\Gamma_n=\Gamma$ is connected, $\Gamma_{n-1}$ has at most two components.

\textbf{The base case:} Observe that the edge $e_1$ is necessarily invariant, since $\Gamma_0=\emptyset$.  Suppose 
$\Gamma_1$ has a single vertex. Then $\pi_1\Gamma_1\cong\integers$, i.e. the mapping torus of $f|_{\Gamma_1}$ is a torus.  
On the other hand, if $\Gamma_1$ has two vertices, then $e_2$ is invariant, because $p_1\to \Gamma_1$ is immersed and 
therefore trivial.  Continuing in this way, we eventually find that there exists $i_0\ge1$ so that, for some 
$\alpha\in\Lambda_{i_0}$, the component $\Gamma_{i_0}^\alpha$ is non-simply-connected and each edge of 
$\Gamma_{i_0}^\alpha$ is invariant. 

Writing $G=\langle \pi_1\Gamma, t\mid \{tft^{-1}=\Phi^k(f):f\in\pi_1\Gamma\}\rangle$, we thus have a subgroup 
$G_0\cong\pi_1\Gamma_{i_0}^\alpha\times\langle t\rangle$ in $G$.  Since  $\pi_1\Gamma_{i_0}^\alpha\ne\{1\}$, the group $G_0$ 
decomposes as the direct product of two infinite groups, so $G_0$ is strongly thick of order $0$.  Hence, if $i_0=n$, then 
on one hand, $\Phi$ has polynomial growth of order $0$, and on the other hand, $G_0=\pi_1\Gamma_{i_0}\times\langle t\rangle$ 
is thick of order $0$, as required.

\textbf{The iterated splitting:}  Suppose that $n>i_0$.  By construction, $\Gamma_n=\Gamma_{n-1}\cup e_n$.  Since 
$\Gamma_{n-1}$ is $f$--invariant and $f(e_n)=e_np_n$, we have an associated splitting of $G_n$ as a graph of groups with the 
following properties:
\begin{itemize}
\item The underlying graph is a single edge.
\item The vertex groups have the form $\pi_1\Gamma_{n-1}^\alpha\rtimes_{\Phi^k}\integers$.  Note that $|\Lambda_{n-1}|$ is 
$1$ or $2$ according to whether the edge $e$ separates $\Gamma_n$.  (Recall that $\Gamma_n=\Gamma$ is connected, and the 
open edge $e$ has $1$ or $2$ complementary components.)

Moreover, at most one component of $\Gamma_{n-1}$ is simply connected.  (If not, $\pi_1\Gamma_n\cong\{1\}$, 
contradicting that $\rank(F)\ge1$.)
\item The edge-groups are conjugate to $\langle t\rangle$.  
\end{itemize}

Viewed as a graph of spaces, the mapping torus $M_n$ of $f$ has vertex spaces which are mapping tori of the restriction of 
$f$ to components of $\Gamma_{n-1}$.  We are attaching a cylinder as follows.  First, in $M_{n-1}$, every edge not belonging 
to $\Gamma_{n-1}$ is a \emph{horizontal} edge that joins some $v\in\Gamma_{n-1}$ to itself and, viewed as a loop in $M_n$, 
represents a conjugate of $t$.  Our cylinder is attached on one side along a horizontal edge.  On the other side, it is 
attached along a path of the form $p_nt_n$, where $p_n$ is as above, and $t_n$ is a horizontal edge.

(This splitting is discussed in detail in~\cite[Section 3]{Macura}, where it is called the \emph{topmost edge 
decomposition}.  The only difference is that, for the moment, we are just removing the single edge $e_n$, rather than many 
edges, as Macura does in producing the topmost edge decomposition.) 

For each component $\Gamma_{n-1}^\alpha$ in $\Lambda_{n-1}$, there is an induced filtration of $\Gamma_{n-1}^\alpha$ so that 
$\Gamma_{n-1}^\alpha$ and the restriction of $f$ to $\Gamma_{n-1}^\alpha$ satisfy 
properties~\eqref{item:filtration},\eqref{item:1_edge},\eqref{item:path} above.  By induction on the number of edges, either 
the vertex group   $\pi_1\Gamma_{n-1}^\alpha\rtimes_{\Phi^k}\integers$ is thick of order at most $n-1-i_0$, or, if 
$\Gamma_{n-1}^\alpha$ is simply connected, then $\pi_1\Gamma_{n-1}^\alpha\rtimes_{\Phi^k}\integers$ is isomorphic to an 
incident edge group.  There is at least one $\Gamma_{n-1}^\alpha$ so that the former holds.

We now check that the vertex group is quasi-convex in the sense of~\cite{BD:thick}, described above.  To this end, note that 
the map $F\rtimes_\Phi\langle t\rangle\to\langle t\rangle$ is a coarsely Lipschitz retraction to $\langle t\rangle$, so 
that, regarding a Cayley graph of $G$ as a tree of spaces associated to the above splitting, we have that each edge-space is 
a coarsely Lipschitz retract.  Fix a  vertex space $V$, and fix $x,y\in V$.  Let $\gamma$ be a geodesic of $G$ joining 
$x,y$.  Then either $\gamma$ lies in $V$, or we can write $\gamma=\alpha_0\beta_1\alpha_1\cdots\beta_k\alpha_k$, where each 
$\alpha_i$ lies in $V$ and each $\beta_i$ starts and ends in some edge space $E_i$ incident to $V$.  Replacing each $\beta_i$ 
by its projection to $E_i$ gives a path in $V$ that joins $x$ to $y$ and has length bounded by a linear function of 
$\dist_G(x,y)$.  Hence $V\hookrightarrow G$ is a quasi-isometric embedding, so 
any geodesic in $V$ (which is a connected graph) from $x$ to $y$ maps to a 
(uniform-quality) quasigeodesic of $G$ that lies in $V$.

Hence, by the inductive hypothesis and~\cite[Proposition 4.4]{BD:thick}, $G_n$ is strongly thick of order $\tau_n\leq 
n-i_0$.  This 
completes the proof that $G_n$ is thick and thus proves 
Corollary~\ref{cor:minimal_relhyp}; we now bound the order of thickness 
independently of the relative train track 
representative.

\textbf{Upper bound on order of thickness:}  We now analyse related splittings of $G$ to 
bound the order of thickness $\tau_n$ in terms of the polynomial growth rate $\eta_n$. For $n=i_0$, we 
saw that $\tau_n=\eta_n=0$, and we are done.  

Suppose $n>i_0$.  Recall that each edge $e_b$ has an associated \emph{polynomial growth rate}~\cite[Definition 
2.11]{Macura}.  Let $d_i$ be the polynomial growth rate of the edge $e_i$.  At this point, we will also apply 
Proposition~2.7 from~\cite{Macura} in order to assume that $f$ is a \emph{Kolchin map}.  The exact definition is not 
important, but this assumption will enable us to use facts from~\cite{Macura} about growth rates of edges.  

There are two cases.  First, suppose that $\eta_n\ge2$.  Then by Lemma~2.16 of~\cite{Macura}, there is a nonempty set 
$\mathcal E$ containing exactly the edges $e_i$ with $d_i=\eta_n$. By the same 
lemma, each $e_i\in\mathcal E$ maps over some edge 
$e_{j(i)}$ with $d_{j(i)}=\eta_n-1$, and conversely any edge mapping over some edge of growth rate $\eta_n-1$ belongs to 
$\mathcal E$.  An edge $e_i$ is \emph{doomed} if either $e_i\in\mathcal E$, or the following holds: $d_i<\eta_n$, and the 
largest connected subgraph of $\Gamma$ that contains $e_i$ and consists of edges not in $\mathcal E$ is simply-connected.  A 
vertex is doomed if each of its incident edges is doomed.

Observe that if $e$ is a doomed edge not in $\mathcal E$, then $e$ is 
invariant.  Indeed, the image of $e$ has the form $ep$, where $p$ is an 
immersed path consisting of edges in the subgraph of $\Gamma$ described above, 
which is simply-connected since $e$ is doomed.  Thus $p$ is trivial, i.e. $e$ 
is invariant.

Let $\Gamma'$ be obtained from $\Gamma$ by removing each doomed vertex and removing the interior of each doomed edge.  Then 
each component of $\Gamma'$ is $f$--invariant (note that its constituent edges need not be invariant).  Indeed, let $e_i$ be 
an edge of $\Gamma'$.  Then $d_i\leq \eta_n-1$, so by Lemma~2.16 of~\cite{Macura}, the path $p_i$ cannot traverse any edge 
in $\mathcal E$.  Now, suppose that $p_i$ has a subpath $q$ lying in a 
connected subgraph $C$ that is maximal with 
the property that none of its edges is in $\mathcal E$.  Write $p_i=aqb$.  Then since $a,b$ cannot contain edges in 
$\mathcal E$, maximality of $C$ implies that $a,b$ lie in $C$, so $p_i$ is an immersed closed path in $C$.  Hence $C$ is not 
simply-connected, so its edges are not doomed.  Thus $p_i$, and hence $f(e_i)=e_ip_i$, lies in $\Gamma'$.
%

Thus, removing the edges of $\mathcal E$ induces a splitting of $G$ as a graph of groups whose edge groups are infinite 
cyclic and whose vertex groups are (necessarily quasi-convex) subgroups which are either: (a) mapping tori of 
the restriction of $\Phi^k$ to subgroups on which $\Phi^k$ has growth rate at most $\eta_n-1$; or (b) conjugate to edge 
groups.  Hence, by induction and~\cite[Proposition~4.4]{BD:thick}, we have that $G$ is thick of order at most 
$\eta_n-1+\tau_{1}$: here, $\tau_{1}$ is the maximal order of thickness of $\pi_1\Lambda\rtimes_{\Phi^k}\integers$, where 
$\Lambda$ is a connected, non-simply connected $f$--invariant subgraph of $\Gamma$ consisting entirely of edges whose 
polynomial growth rates are at most $1$.

Hence it remains to consider the case where $\eta_n=1$.  Here the situation is somewhat more complicated because~\cite[Lemma 
2.16]{Macura} does not apply: linearly-growing edges can map over other linearly-growing edges.  This case is handled in 
Lemma~\ref{lem:linear_case}, which shows that $\eta_n=1$ in the linearly growing case.  So, $\tau_1=1$ and $G$ is thick of 
order at most $\eta$.

\textbf{Lower bound on order of thickness:}  By~\cite[Theorem 8.1]{Macura}, the divergence function of $G$ is 
polynomial of degree at least $\eta+1$.  On the other hand, if $G$ is strongly thick of order $\tau$, then 
by~\cite[Corollary 
4.17]{BD:thick}, $G$ has divergence function that is polynomial of degree at most $\tau+1$.  This gives a 
contradiction unless $\tau\geq\eta$.  We thus conclude that $G$ is strongly thick of order $\eta$.

More precisely, we have the following: given $x\in G$ and $r\geq 0$, and $y,z\in G$ with $\dist_G(x,y),\dist_G(x,z)\leq r$, 
let $\mu_x(y,z)$ be the infimum of $|P|$, where $P$ varies over all paths in $G$ from $y$ to $z$ that avoid the ball of 
radius $r/2$ about $x$.  Let $\chi_x(r)$ be the supremum of $\mu_x(y,z)$ over all such $y,z$, and let $\chi(r)$ be the 
supremum of $\chi_x(r)$ over all $x\in G$.  Theorem~8.1 of~\cite{Macura} shows that $\chi(r)$ is bounded below by a 
polynomial  of degree $\eta+1$.

On the other hand, applying Theorem~4.9 of~\cite{BD:thick} inductively, exactly as in the proof of~\cite[Corollary 
4.17]{BD:thick}, shows that $\chi(r)$ is bounded above by a polynomial of degree $\tau+1$.  The only difference between our 
situation and that in~\cite{BD:thick} is the base case.  Specifically, in order to apply~\cite[Theorem 8.2]{Macura}, we 
defined $\chi(r)$ using paths that avoid balls of radius $r/2$.  When applying~\cite[Theorem 4.9]{BD:thick}, we are taking 
advantage of the fact that the constant $\delta$ in that statement can be any element of $(0,1)$; we are using the case 
$\delta=\frac12$.  

In the base case, we cannot rely on~\cite[Proposition 4.12]{BD:thick}, as is done in the proof of~\cite[Corollary 
4.17]{BD:thick}, because that statement requires $\delta\in(0,\frac{1}{54})$.  Instead, we use a simple special case 
of~\cite[Proposition 4.12]{BD:thick}.  Specifically, we need to show that there is a fixed linear function $\rho$ so that 
for each component $\Lambda$ of $\Gamma_{i_0}$, the mapping torus $M_\Lambda$ of the restriction of $f$ to $\Gamma_{i_0}$ 
satisfies $\chi(r)\leq\rho(r)$, where $\chi$ is now defined in $\pi_1M_\Lambda$.  But this is clear since 
$M_\Lambda=\Lambda\times\mathbb S^1$, and we can take $\rho(r)=4r$.
\end{proof}

It remains to prove thickness of order $1$ in the linear growth case:

\begin{lem}\label{lem:linear_case}
Suppose that the automorphism $\phi$ has linear growth.  Then $G$ is strongly thick of order at most $1$.
\end{lem}

\begin{proof}
We will induct on the number of edges in the filtration of a graph $\Gamma$ associated to a relative train track 
representative of $\phi$.  The goal is to construct a \emph{tight network} $\mathcal W$ of \emph{uniformly wide} subspaces 
of $G$; as in~\cite{BD:thick}, this is sufficient to prove the claim.  (The definitions of \emph{tight network} and 
\emph{uniformly wide} are given below.)

\textbf{The relative train track representative:}  Exactly as in the proof of Theorem~\ref{thm:main}, we can assume that 
$\phi$ is represented by a Kolchin map 
$f:\Gamma\to\Gamma$ with $\Gamma$ equipped with a filtration $\Gamma_0\subset\cdots\subseteq\Gamma_n=\Gamma$ exactly as 
before.  

We are going to use the following properties of $f$:
\begin{enumerate}[(i)]
     \item \label{item:nielsen}For each $i$, we have $f(e_i)=e_ip_i$ where $p_i$ is either trivial (i.e. $e_i$ is an 
invariant edge) or $p_i$ is an immersed closed (possibly trivial) \emph{Nielsen path},  i.e. the tightening of $f(p_i)$ is 
$p_i$.    This occurs when $e_i$ is a linearly-growing edge by~\cite{BFH:dynamics}.

    \item \label{item:order}If $e_i$ is an 
invariant edge, and $e_j$ is an edge with $j<i$, then we can reverse the 
order of $e_i,e_j$ in the filtration, because $p_i$ cannot map over $e_j$, since $e_i$ is invariant.  Hence we can 
assume that $d_n=1$, where $e_n$ is the topmost edge in the filtration.
\end{enumerate}

\textbf{Sub-mapping tori:}  Let $M_n$ be the mapping torus of $f$.  For each $i\leq n$, and each component $\Gamma^\alpha_i$ 
of $\Gamma_i$,  we have 
that $\Gamma^\alpha_i$ is $f$--invariant, and we take $M_i^\alpha$ to be the mapping torus of the restriction of $f$ to 
$\Gamma^\alpha_i$.  Note that $M_i^\alpha\subset M_j^{\alpha'}$, for some $\alpha'$, whenever $i<j$.

\textbf{Collapsing cyclic sub-mapping tori to circles:} Let $i_0\leq n$ be maximal such that each component 
$\Gamma_{i_0}^\alpha$ of $\Gamma_{i_0}$ consists of $f$--invariant 
edges.  So, property~\eqref{item:order} guarantees that $d_i=1$ if and only if $i>i_0$.  We may assume that $n>i_0$, for 
otherwise $G$ is thick of order $0$, and we are done.

We now do some collapsing; this step isn't necessary, but makes later parts of the proof easier to picture.  We would like 
to assume that every component $\Gamma_{i_0}^\alpha$ consisting of invariant edges is a core graph, i.e. has no valence--$1$ 
vertex.  To this end, collapse each free face of $\Gamma_{i_0}^\alpha$, yielding a new graph $\bar\Gamma_n$ which is a 
deformation retract of $\Gamma_n$.  The map $f$ descends to a homotopy equivalence $\bar f:\bar\Gamma_n\to\bar\Gamma_n$ 
inducing the map $\phi$ on fundamental group, because we collapsed invariant edges.  There is an obvious filtration of $\bar 
\Gamma_n$ induced by the filtration of $\Gamma_n$, and it is easily verified that $\bar f$ satisfies 
properties~\eqref{item:nielsen},\eqref{item:order} above.  

So, we can assume that the components 
$\Gamma_{i_0}^\alpha$ of $\Gamma_{i_0}$ are either single points or 
non-simply connected core graphs; as in the proof of Theorem~\ref{thm:main}, there exists at 
least one component of the latter type.  

\textbf{The sub-network $\mathcal W_0$:}  Let $\rho:\widetilde M_n\to M_n$ be the universal cover.  Let $\mathcal W_0$ be 
the set of components of $\rho^{-1}(M_{i_0}^\alpha)$, where $\alpha\in\Lambda_{i_0}$ is such that the component 
$\Gamma_{i_0}^\alpha$ of $\Gamma_{i_0}$ is non-simply-connected.  So, each $W\in  \mathcal W_0$ is isometric to the product 
of $\reals$ with one of finitely many trees.

\textbf{The sub-network $\mathcal W_1$ of tori:}  For each $i_0<j\leq n$, consider the immersed closed Nielsen path 
$p_j\to\Gamma_{j-1}$.  Let $v_j$ be the initial vertex of $p_j$ and let $t_j$ be the (unique) edge of $M_{j-1}$ joining 
$v_j$ to itself to produce a loop representing a conjugate of $\langle t\rangle$ in $\pi_1M_n$.  Situating the basepoint of 
$M_n$ at $v_j$, we see that the elements of $\pi_1M_n$ represented by $p_j$ and $t_j$ commute, so we have a torus $T_j$ and 
a $\pi_1$--injective cellular map $T_j\to M_n$ so that the image of the induced map $\pi_1T_j\to\pi_1(M_n,v_j)$ is the 
$\integers^2$ subgroup generated by $t_j$ and $p_j$.   We can choose $T_j$ to lie in some component 
$M_{j-1}^\alpha$ of $M_{j-1}$, namely the component containing $v_j$, and that the paths $t_j\to M_n$ and $p_j\to M_n$ 
factor through $T_j\to M_n$.

\begin{rem}[A $T_j$ example]\label{exmp:torus}
Here is an example of a torus of the type just described.  Let $F=\langle e_0,e_1,e_2,e_3\mid\rangle$ and let $f$ be defined 
by: $f(e_0)=e_0,f(e_1)=e_1e_0,f(e_2)=e_2e_0,f(e_3)=e_3e_0e_1e_2^{-1}$.  So, $$G=\langle e_0,e_1,e_2,e_3,t\mid 
te_0t^{-1}=e_0,te_1t^{-1}=e_1e_0,te_2t^{-1}=e_2e_0,te_3t^{-1}=e_3e_0e_1e_2^{-1}\rangle.$$  The tori $T_1,T_2$ 
correspond to the subgroup $\langle t,e_0\rangle$.  The torus $T_3$ corresponds to the subgroup $\langle 
t,e_0e_1e_2^{-1}\rangle$.  Each of these is a $\integers^2$ subgroup.  For example, note that 
$te_0e_1e_2^{-1}t^{-1}=f(e_0e_1e_2^{-1})=e_0e_1e_0e_0^{-1}e_2^{-1}=e_0e_1e_2^{-1}.$

Consider the topmost edge splitting obtained by removing $e_3$.  This is an HNN extension with vertex group 
$\langle e_0,e_1,e_2,t\rangle$ and stable letter $e_3$; we have $e_3^{-1}te_3=e_0e_1e_2^{-1}t$.  Note 
that $\langle 
e_0e_1e_2^{-1}t\rangle$ is contained in the subgroup $\langle t,e_0e_1e_2^{-1}\rangle$ 
corresponding to $T_3$.  This phenomenon will be important later.  This concludes the example, and we resume the proof.
\end{rem}

\textbf{The candidate network:}  Let $\mathcal W_1$ be the set of components of $\rho^{-1}(T_j)$ for $i_0<j\leq n$.  Let 
$\mathcal 
W=\mathcal W_0\cup\mathcal W_1$ and note that the elements of $\mathcal W$ coarsely cover $\widetilde M$.

\textbf{Quasiconvexity of the elements of $\mathcal W$:}  Arguing as in the proof of Theorem~\ref{thm:main} shows that 
$\widetilde M_i^\alpha$ is (uniformly) quasiconvex in $\widetilde M^{\alpha'}_{i+1}$ whenever $M_i^\alpha \subset 
M_{i+1}^{\alpha'}$, so each element of $\mathcal W_0$ is $(r,r)$--quasiconvex in $\widetilde M_n$ for some fixed constant 
$r$.  (Here quasiconvexity is in the sense of~\cite{BD:thick}.)

Let $\widetilde T_j\in\mathcal W_1$.  Then $\widetilde T_j$ uniformly coarsely coincides with the orbit of an $\integers^2$ 
subgroup of $\pi_1M_{j-1}^\alpha$, which is a free-by-cyclic group.  By Corollary 6.9 of~\cite{Button:unipotent}, abelian 
subgroups of free-by-$\integers$ groups are undistorted, so any two points in $\widetilde T_j$ can be joined by a 
uniform-quality discrete quasigeodesic of $\pi_1M_{j-1}^\alpha$ (hence a quasigeodesic of $\pi_1M_n$) that lies in 
$\widetilde T_j$.  Since there are only finitely many orbits of subspaces in $\mathcal W$, we conclude that there exists 
$s$ such that each $W\in\mathcal W$ is $(s,s)$--quasiconvex in $\widetilde M_n$ (quasiconvexity in the sense 
of~\cite{BD:thick}).

%

\textbf{$\mathcal W$ is uniformly wide:}  $\mathcal W$ contains finitely many isometry types of spaces, each of which is 
quasi-isometric to $F\times\integers$ for some finitely generated free group $F$.  So, any ultralimit of rescaled spaces in 
$\mathcal W$ has no cut-point, i.e. $\mathcal W$ is \emph{uniformly wide} in the sense of~\cite[Definition 
4.11]{BD:thick}.

\textbf{Induction:}  Fix $i\geq i_0$ and $r\ge0$.  By induction on the number of edges in $\Gamma_i$, there exists 
$\ell(i-1,r)$ such that the following hold for each $M^\alpha_{i-1}$ which is not a circle, i.e. where 
$\Gamma^\alpha_{i-1}$ is not a point (the first property is the defining property of a tight network, and the second 
property will be needed to make the induction work):
\begin{enumerate}
     \item \label{item:network}  Let 
$W,W'\in\mathcal W$ be contained in $\widetilde M^\alpha_{i-1}$.  Suppose that, for some $x\in\widetilde M^\alpha_{i-1}$, 
each of $W,W'$ intersects $\neb_r(x)$.  Then there is a sequence $W=W_1,\ldots,W_{\ell(i-1,r)}=W'$ such that each 
$W_t\in\mathcal W$, each $W_t$ lies in $\widetilde M^\alpha_{i-1}$, and for all $t\leq\ell(i-1,r)-1$, the spaces 
$W_t,W_{t+1}$ have unbounded, coarsely connected coarse intersection.
     
    \item\label{item:next_torus} Let $\rho:\widetilde M^\alpha_{i-1}\to M^\alpha_{i-1}$ be the universal covering map.  
Let $M^\beta_i$ be the component of $M_i$ containing $M^\alpha_{i-1}$.  Suppose that $e_i\subset M_i^\beta$.  Consider the 
graph of spaces decomposition of $M_i^\beta$ induced by removing $e_i$ from $\Gamma_i^\beta$ (i.e. the topmost edge 
decomposition).  Then the edge space is a circle intersecting $\Gamma^i_\beta$ in the midpoint of $e_i$.  Let 
$E\subset\widetilde M^\alpha_{i-1}$ be a component of the $\rho$--preimage of this circle.  Then $E$ is uniformly coarsely 
contained in some $W\in\mathcal W$ that lies in $\widetilde M^\alpha_{i-1}$.
%
%
\end{enumerate}

In the base case, where $i=i_0$, the former statement holds since $\widetilde M^\alpha_{i_0}\in\mathcal W$.

Fix $M^\beta_i$, the mapping torus of the restriction to $f$ of some 
non-simply connected $\Gamma_i^\beta$.  We will verify that the same properties 
hold for $\widetilde M^\beta_i$.  First, as before, removing the 
edge $e_i$ from $\Gamma^\beta_i$ decomposes $\widetilde M^\beta_i$ as a tree $\mathcal T$ of spaces whose vertex spaces are 
various translates of various $\widetilde M^\alpha_{i-1}$ and whose edge spaces are two-ended.  

Consider a splitting of $M^\gamma_{i+1}$ that arises from the deletion of $e_{i+1}$ from $\Gamma_{i+1}$ and has $M^\beta_i$ 
as a vertex space. The incoming edge space is attached via a circle in $M^\beta_i$ homotopic into the $\rho$--image of some 
element of $\mathcal W$ contained in $\widetilde M^\beta_i$.

Indeed, let $C$ be such an attaching circle in $M^\beta_i$.  There are two cases.  First, we could have   
$C=p_{i+1}t_{i+1}$, where $t_{i+1}$ is 
an edge of $M^\beta_i$ that forms a loop representing a conjugate of $t$, and $p_{i+1}$ is the Nielsen path such that 
$f(e_{i+1})=e_{i+1}p_{i+1}$.  In this case, the torus 
$T_{i+1}$ contains $t_{i+1}$ and the image of $p_{i+1}$, by construction.

Otherwise, $C$ traverses a single horizontal edge $t'$ in $M^\beta_i$ joining some vertex $v\in M^\beta_i$ to itself.  If 
$v$ is contained in some nontrivial $\Gamma^\delta_{i_0}\subset M^\beta_i$, then $t'$ is contained in $M^\delta_{i_0}$, and 
we are done, because then $\widetilde M^\delta_{i_0}\in\mathcal W$.  Otherwise, every edge of $\Gamma^\beta_{i}$ containing 
$v$ has the form $e_j$ for some $j\leq i$ (here we have used that nontrivial components of $\Gamma_{i_0}$ are core graphs).  
Applying property~\eqref{item:next_torus} inductively, $\widetilde C$ lies at Hausdorff distance at most $1$ from a line that 
is coarsely contained in some element of $\mathcal W$ lying in $\widetilde M^\beta_i$.  This verifies 
property~\eqref{item:next_torus} for $M^\beta_i$.  (The constant implicit in ``coarsely contained'' has increased, but this 
happens at most $n$ times.)

Now we verify property~\eqref{item:network}.  Let $x\in\widetilde M_i^\beta$ lie at distance at most $r$ from 
$W,W'\in\mathcal W$.  Let $v,v'$ be vertices of $\mathcal T$ so that the corresponding vertex spaces $V,V'$ have unbounded 
intersection with $W,W'$ respectively.  Let $\gamma$ be a geodesic of $\mathcal T$ from $v$ to $v'$, so that $|\gamma|\leq 
2r$.  For each valence-$2$ vertex of $\gamma$ whose corresponding vertex space 
is the mapping torus of the restriction of $f$ to a non-simply-connected graph, 
let $A,B\in\mathcal W$ lie in the associated vertex space and 
respectively coarsely contain the incoming and outgoing edge spaces along $\gamma$.  By the inductive hypothesis, $A,B$ can 
be joined by a sequence of at most $\ell(i-1,Kr)$ elements of $\mathcal W$ that lie in the given vertex space, where $Kr$ is 
the distance in the vertex space from $A$ to $B$ (the constant $K$ depends only on $s$).  Consecutive elements of the 
sequence have unbounded, coarsely connected coarse intersection.

Similarly, if $v$ is the initial vertex of $\gamma$, then we can choose $B\in\mathcal W$ that lies in the associated vertex 
space, contains the initial edge space, and can be joined to $W$ by a sequence of at most $\ell(i-1,Kr)$ elements of 
$\mathcal W$, with the same intersection properties.  The same holds for the terminal vertex of $\gamma$, with $W'$ replacing 
$W$.  Hence property~\eqref{item:network} holds, with $\ell(i,r)=\ell(i-1,Kr)(2r+1)$.  

\textbf{Conclusion:}  We have shown that $\mathcal W$ is a uniformly wide \emph{tight network} in $\widetilde M_n$, so, 
according 
to~\cite[Definition 4.13]{BD:thick}, $\widetilde M_n$, and hence $G$, is 
strongly thick of order at most $1$.  (Indeed, to prove that $\mathcal W$ is a 
tight network, it sufficed to exhibit the constant $\ell(n,3s)$.)
\end{proof}

\bibliographystyle{alpha}
\bibliography{thick_free_by_z}
\end{document}